\documentclass{article}

\usepackage{amsmath} 
\usepackage{amsthm}
\usepackage{amssymb}
\usepackage{hyperref} 
\usepackage{color}
\usepackage{amstext,dsfont,fancyvrb,float,fontenc,graphicx, subcaption}
\usepackage{mathrsfs}

\newtheorem{theorem}{Theorem}
\newtheorem{lemma}[theorem]{Lemma}
\newtheorem{proposition}[theorem]{Proposition} 
\newtheorem*{corollary}{Corollary} 

\theoremstyle{definition}
\newtheorem*{definition}{Definition} 

\theoremstyle{remark}
\newtheorem*{remark}{Remark}

\usepackage{tikz}
\usetikzlibrary{calc, patterns, arrows.meta, decorations.markings}
\tikzset{%
  line numbers/.store in=\fakelinenos,
  line numbers=50,
  line number shift/.store in=\fakelinenoshift,
  line number shift=5mm,
line number style/.style={text=black},
}

\title{Polygonal Refractive Outer Billiards}
\author{Jaewoo Park\\ University Laboratory High School
}

\begin{document}
\maketitle

\begin{abstract}
Extending recent work on \textit{refractive billiards}, we introduce and study the corresponding \textit{refractive outer billiards} system about a convex polygon. Gutkin and Simányi showed in 1992 that for regular outer billiards, orbits about a certain class of polygons called \textit{quasi-rational} polygons are bounded, and that orbits about \textit{rational} polygons are periodic. Tabachnikov and Culter later proved in 2007 that every outer billiard system about a convex polygon admits a periodic trajectory. We generalize both results to the refractive setting.
\end{abstract}

\section{Introduction}
Let $\Gamma$ be a closed convex curve in the plane. The \textit{outer billiards map} $T$  around $\Gamma$ is defined as follows. Let $E$ be the region outside of $\Gamma$, and let $o \in E$ be a point in the plane. There are two \textit{supporting rays} $L_+(o)$ and $L_-(o)$ from $o$ to $\Gamma$, such that the entirety of $\Gamma$ lies to the right of $L_+$ and to the left of $L_-$. If $L_+$ has a unique intersection with $\Gamma$, then $T$ is a reflection about the point of tangency. If the intersection is \textit{not} unique, then $T$ is undefined. See Figure~\ref{fig:smoothTable}.

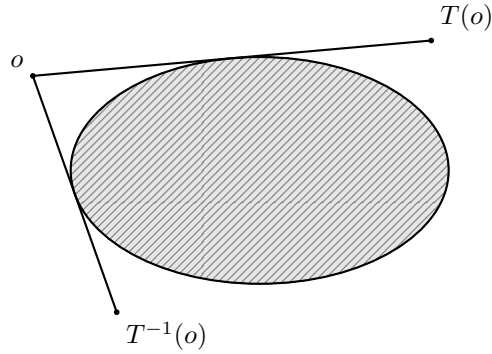
\begin{figure}[!htb]
    \centering
\begin{tikzpicture}[scale=2.5]
  \pgfmathsetmacro{\a}{1.0}
  \pgfmathsetmacro{\b}{0.6}

  \pgfmathsetmacro{\ox}{-1.2}
  \pgfmathsetmacro{\oy}{0.5}
  \coordinate (o) at (\ox,\oy);

  \fill[gray!20] (0,0) ellipse [x radius=\a, y radius=\b];
  \draw[pattern=north east lines, pattern color=gray]
        (0,0) ellipse [x radius=\a, y radius=\b];
  \draw[thick] (0,0) ellipse [x radius=\a, y radius=\b];

  \pgfmathsetmacro{\p}{\ox/\a}
  \pgfmathsetmacro{\q}{\oy/\b}

  \pgfmathsetmacro{\R}{sqrt(\p*\p + \q*\q)}

  \pgfmathsetmacro{\phi}{atan2(\q,\p)}

  \pgfmathsetmacro{\alpha}{acos(1/\R)}

  \pgfmathsetmacro{\tOne}{\phi + \alpha}
  \pgfmathsetmacro{\tTwo}{\phi - \alpha}

  \pgfmathsetmacro{\xone}{\a*cos(\tOne)}
  \pgfmathsetmacro{\yone}{\b*sin(\tOne)}
  \pgfmathsetmacro{\xtwo}{\a*cos(\tTwo)}
  \pgfmathsetmacro{\ytwo}{\b*sin(\tTwo)}

  \pgfmathsetmacro{\Ax}{ifthenelse(\yone > \ytwo, \xone, \xtwo)}
  \pgfmathsetmacro{\Ay}{ifthenelse(\yone > \ytwo, \yone, \ytwo)}
  \pgfmathsetmacro{\Apx}{ifthenelse(\yone > \ytwo, \xtwo, \xone)}
  \pgfmathsetmacro{\Apy}{ifthenelse(\yone > \ytwo, \ytwo, \yone)}

  \coordinate (A)  at (\Ax,\Ay);
  \coordinate (A') at (\Apx,\Apy);

  \coordinate (To)   at ($ 2*(A)  - (o) $);
  \coordinate (Tinv) at ($ 2*(A') - (o) $);

  \draw[thick] (o) -- (To);
  \draw[thick] (o) -- (Tinv);

  \fill (o)    circle (0.015) node[above left]  {$o$};
  \fill (To)   circle (0.015) node[above right] {$T(o)$};
  \fill (Tinv) circle (0.015) node[below right] {$T^{-1}(o)$};
\end{tikzpicture}
\caption{The outer billiards map $T$.}
    \label{fig:smoothTable}
\end{figure}

Moser first popularized outer billiards as a toy model for the solar system in \cite{Moser}. Problems about the stability and structure of orbits are of great interest: for instance, Moser-Douady showed that if $\Gamma$ is $C^7$-curved and has positive curvature, then all orbits are bounded~\cite{Moser2, Douady}. Dolgopyat and Fayad later proved that outer billiards around a semicircle has unbounded orbits~\cite{Dolgopyat}.

Our interest lies in the case where $\Gamma$ is a convex polygon; therefore, we will now use $P$ instead of $\Gamma$. The study of orbits about polygons is rich and complex. Gutkin-Simányi~\cite{Gutkin}, Kolodziej~\cite{Kolodziej}, and Vivaldi-Shaidenko~\cite{Vivaldi} independently proved that all orbits about a class of polygons called \textit{quasi-rational polygons} are bounded, and that all orbits about \textit{rational polygons} are periodic\footnote{Gutkin and Sim\'anyi remark in \cite{Gutkin} that they were unable to verify the arguments of \cite{Kolodziej, Vivaldi}.}. All rational polygons are quasi-rational; for instance, all regular $n$-gons are quasi-rational, but only the regular 3, 4 and 6-gons are rational. 
Tabachnikov and Culter~\cite{Culter} also showed that every polygon has periodic orbits and that the points that induce periodic orbits have a positive measure in the plane. For other notable results, see \cite{Genin, Schwarz, Boyland}.

Recently, in \cite{Arnold}, we generalized regular (Birkhoff) billiards to the \textit{refractive billiards} system. In this paper, we build on our paper, introducing the \textit{refractive outer billiards} system. We generalize the main results of \cite{Gutkin} and \cite{Culter}.

The refractive outer billiard map $T$ around a convex $n$-gon $P$ is defined as follows. Throughout this paper, we follow Gutkin's notation. Let $\lambda_1, \lambda_2, \dots, \lambda_k$ be a sequence of positive \textit{refractive indices} (viewed modulo $k$) that satisfy 
\[
\prod_{i=1}^{k} \lambda_i = 1.
\]
Given a point $o$ outside $P$, the map $T$ is well-defined (with respect to all choices of refractive indices) when $o$ does not lie on a line through a side of $P$. If this is the case, let $A_1$ denote the intersection of $L_+(o)$ with $P$. Then $T(o)$ is the reflection of $o$ about $A_1$, followed by a scaling of $\lambda_1$ about $A_1$. Generally, $T^m(o)$ is the reflection of $T^{m-1}(o)$ about $A_m$, followed by a scaling of $\lambda_m$ about $A_m$. Note that the $m$-th step carries the index $\lambda_m$. All statements
about orbits and periodicity refer to the bi-infinite sequence $(T^m(o))_{m \in \mathbb{Z}}$ together with its indices. See Figure~\ref{fig: refractivedualbilliards}.

\begin{figure}[!htb]
\centering
\begin{tikzpicture}[scale=3.5]
  \coordinate (A)  at (0,0);
  \coordinate (B)  at (1,0);
  \coordinate (C)  at (0.8,0.9);
  \coordinate (D)  at (0.15,0.8);
  \coordinate (P0) at (0.4,-0.35);
  \pgfmathsetmacro{\lA}{3/2}   
  \pgfmathsetmacro{\lD}{2/3}   
  \pgfmathsetmacro{\lC}{2}     
  \pgfmathsetmacro{\lB}{1/2}   

  \coordinate (P1) at ($ {(1+\lA)}*(A) + {(-\lA)}*(P0) $);
  \coordinate (P2) at ($ {(1+\lD)}*(D) + {(-\lD)}*(P1) $);
  \coordinate (P3) at ($ {(1+\lC)}*(C) + {(-\lC)}*(P2) $);
  \coordinate (P4) at ($ {(1+\lB)}*(B) + {(-\lB)}*(P3) $);

  \fill[gray!20] (A) -- (B) -- (C) -- (D) -- cycle;
  \draw[pattern=north east lines, pattern color=gray]
        (A) -- (B) -- (C) -- (D) -- cycle;
  \draw[thick] (A) -- (B) -- (C) -- (D) -- cycle;

  \draw[thick]
    (P0) -- (A) -- (P1) -- (D) -- (P2) -- (C) -- (P3) -- (B) -- (P4);

  \fill (P0) circle (0.015) node[below right] {$o$};
  \fill (P1) circle (0.015) node[left]        {$T(o)$};
  \fill (P2) circle (0.015) node[above left]  {$T^2(o)$};
  \fill (P3) circle (0.015) node[right] {$T^3(o)$};
  \fill (P4) circle (0.015) node[below] {$T^4(o)$};
  \fill (A) circle (0.015) node[below left] {$A_1$};
  \fill (B) circle (0.015) node[below right] {$A_4$};
  \fill (C) circle (0.015) node[above right] {$A_3$};
  \fill (D) circle (0.015) node[above left] {$A_2$};
\end{tikzpicture}
\caption{Refractive outer billiards with $\lambda_1 = 3/2,\,\, \lambda_2 = 2/3,\,\, \lambda_3 = 2,\,\, \lambda_4 = 1/2.$}
\label{fig: refractivedualbilliards}
\end{figure}
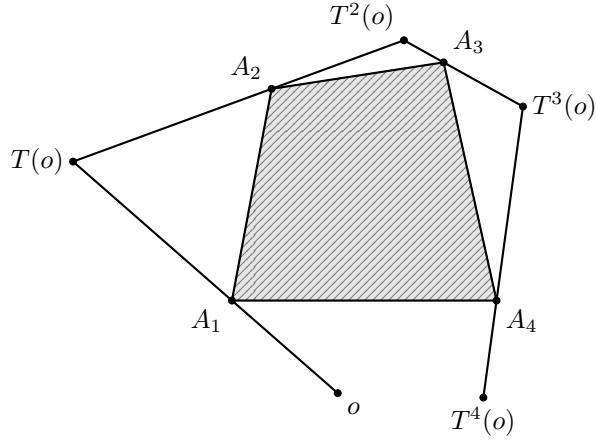

Note that the inverse map $T^{-1}$ is also defined in a similar way. Let $A_0$ denote the intersection of $L_-(o)$ with $P$. Then $T^{-1}(o)$ is the reflection of $o$ about $A_0$, followed by a scaling of $1/\lambda_0$ about $A_0$, and the general case is defined similarly (indices modulo $k$, so $\lambda_0 = \lambda_k$). 

Let $P$ be a convex polygon, $T$ the refractive outer billiard map, and $\lambda_1, \dots, \lambda_k$ refractive indices. Then, our main results are as follows.
\begin{theorem}
    If $P$ is \textit{quasi-rational}, then every orbit of $T$ is bounded.
\end{theorem}
\begin{theorem}
    If $P$ is \textit{rational} and each of $\lambda_1, \dots, \lambda_k$ is rational, then every orbit of $T$ is periodic. 
\end{theorem}
\begin{theorem}
    Given any $P$ and $\lambda_1, \dots, \lambda_k,$ there exists a periodic orbit. 
\end{theorem}

In Section~\ref{sec:basic}, we review Gutkin's construction for outer billiards. We define the necklace map, the cone and ray construction, the necklace polygon, and define what it means for a polygon to be rational and quasi-rational.

In Section~\ref{sec:gutkin}, we prove the generalizations of~\cite{Gutkin} for refractive outer billiards. Using the \textit{refractive necklace map}, we establish that the orbits of the \textit{first return map} for the system is bounded if $P$ is quasi-rational. We show that if $P$ is rational with each refraction coefficient rational, then the possible configurations of $P$ are finite, which implies periodicity.

In Section~\ref{sec:culter}, we prove the generalization of~\cite{Culter} for refractive outer billiards. We use the fact that the refractive necklace map roughly follows Gutkin's necklace polygon.

In Appendix~\ref{sec:motivation}, we present the motivation behind refractive outer billiards, highlighting how the projective duality between billiards and outer billiards extends to the duality between refractive billiards and refractive outer billiards.

\section{Basic Definitions}\label{sec:basic}

In this section, we consider \textit{regular} outer billiards. We summarize and adapt the constructions of Gutkin~\cite{Gutkin} for completeness. Define $\sigma_\ell$ for $\ell\in\mathbb{Z}^+$ to be the set of points in the plane where $T^\ell(o)$ or $T^{-\ell}(o)$ is undefined. For instance, the union of all lines through the sides of $P$ is $\sigma_1$. Inductively, we see that each $\sigma_\ell$ is a finite union of straight lines, so that the set $\sigma = \cup_{n=1}^{\infty}\sigma_n$ is a countable union of straight lines. It follows that $\sigma$ is a set of zero measure in the plane. 

We call the set of points that lie outside of $P$ and are not in $\sigma$ the \textit{strongly regular points} about $P$. For a strongly regular point, $T^\ell$ is well-defined for all integer $\ell.$ From now on, all points that we consider will be strongly regular. For all arguments, the exclusion of a set of measure zero does not impact the logic. 

We now introduce the necklace map, based on a key idea: we reflect the polygons instead of the points. Take $P = P_0$ and a strongly regular point $o.$ Call the left tangency point $A_1$ the \textit{head}, and the right tangency point $A_0$ the \textit{tail}. We then reflect $P_0$ about $A_1$ to get $P_1.$ More generally, a polygon $P_\ell$ has head $A_{\ell+1}$ and tail $A_\ell.$ Reflecting $P_\ell$ about its head gives $P_{\ell+1},$ and reflecting about its tail gives $P_{\ell-1}.$ Given $P_0,$ we can continue this process infinitely in both directions (since $o$ is strongly regular) to obtain a sequence $\{\dots, P_{-1}, P_0, P_1, \dots\}$ called the \textit{necklace of $P$ about $o$.} The \textit{necklace map} sends $P_\ell$ to $P_{\ell+1}.$

We can temporarily forget the outer billiards map  and investigate the necklace instead, due to the following theorem. See Figure \ref{fig: necklace}.
\vspace{0.3cm}

\begin{figure}[!htb]
    \centering
    \begin{tikzpicture}[scale=0.9]
  \coordinate (P0V1) at (0,0);
  \coordinate (P0V2) at (1,0);
  \coordinate (P0V3) at (1,2);
  \coordinate (P0V4) at (-0.5,0.7);

  \coordinate (A0) at (0,0);  
  \coordinate (A1) at (1,2);  
  \coordinate (A2) at (1,4);
  \coordinate (A3) at (1,6);
  \coordinate (O) at (-5,3);
  
  \draw[thick]
    (P0V1) -- (P0V2) -- (P0V3) -- (P0V4) -- cycle;
    \draw[pattern=north east lines, pattern color=gray]
        (P0V1) -- (P0V2) -- (P0V3) -- (P0V4) -- cycle;

  \coordinate (P0mid) at (0.45,0.75);

  \coordinate (P1V1) at ($2*(A1) - (P0V1)$);
  \coordinate (P1V2) at ($2*(A1) - (P0V2)$);
  \coordinate (P1V3) at ($2*(A1) - (P0V3)$);
  \coordinate (P1V4) at ($2*(A1) - (P0V4)$);

  \draw[thick]
    (P1V1) -- (P1V2) -- (P1V3) -- (P1V4) -- cycle;
    \draw[pattern=north east lines, pattern color=gray]
        (P1V1) -- (P1V2) -- (P1V3) -- (P1V4) -- cycle;
  \coordinate (P1mid) at (1.55, 3.25);

  \coordinate (R2) at (1,4);

  \coordinate (P2V1) at ($2*(R2) - (P1V1)$);
  \coordinate (P2V2) at ($2*(R2) - (P1V2)$);
  \coordinate (P2V3) at ($2*(R2) - (P1V3)$);
  \coordinate (P2V4) at ($2*(R2) - (P1V4)$);
  \coordinate (P-1mid) at (-0.4,-0.8);
  \draw[thick]
    (P2V1) -- (P2V2) -- (P2V3) -- (P2V4) -- cycle;
    \draw[pattern=north east lines, pattern color=gray]
        (P2V1) -- (P2V2) -- (P2V3) -- (P2V4) -- cycle;
  \coordinate (P2mid) at (0.45, 4.75);

  \node at (P0mid) {$P_0$};
  \node at (P1mid) {$P_1$};
  \node at (P2mid) {$P_3$};
  \node at (P-1mid) {$P_{-1}$};
  \coordinate (P-1V1) at (-1,0);
  \coordinate (P-1V2) at (0.5, -0.7);
  \coordinate (P-1V3) at (-1,-2);
  
  \draw[thick]
    (A0) -- (P-1V1) -- (P-1V3) -- (P-1V2) -- cycle;
    \draw[pattern=north east lines, pattern color=gray]
        (A0) -- (P-1V1) -- (P-1V3) -- (P-1V2) -- cycle;

  \fill (A0) circle (0.05)
      node[below right, xshift=10pt, yshift=3pt] {$A_0$};
  \fill (A1) circle (0.05) node[below right] {$A_1$};
  \fill (A2) circle (0.05) node[below left] {$A_2$};
  \fill (A3) circle (0.05) node[below right] {$A_3$};
  \fill (P-1V3) circle (0.05) node[left] {$A_{-1}$};
  \fill (O) circle (0.05) node[left] {$o$};
  
  \draw[thick, shorten >=0pt, shorten <=-0.5cm] (A3) -- (O);
\draw[thick, shorten >=0pt, shorten <=-0.5cm] (A2) -- (O);
\draw[thick, shorten >=0pt, shorten <=-0.5cm] (A1) -- (O);
\draw[thick, shorten >=0pt, shorten <=-0.5cm] (A0) -- (O);
\draw[thick, shorten >=0pt, shorten <=-0.5cm] (P-1V3) -- (O);

\end{tikzpicture}
    \caption{Part of the necklace of $P=P_0$ about $o$.}
    \label{fig: necklace}
\end{figure}
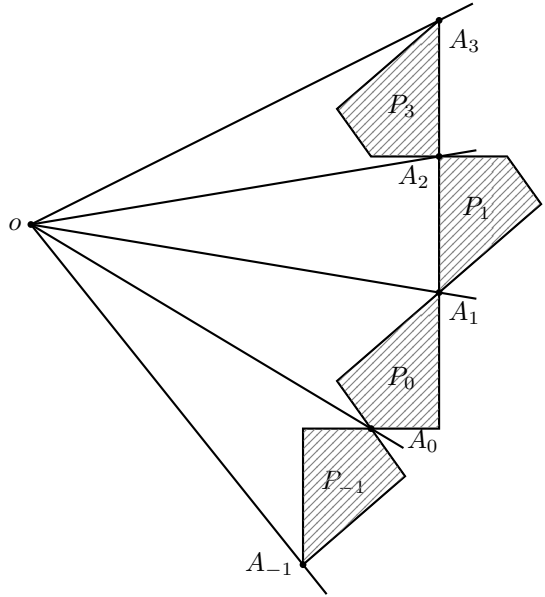

\begin{theorem}[Gutkin \protect{\cite[pp.~434--435]{Gutkin}}]
The necklace of a convex polygon $P$ about $o$ is defined simultaneously with the orbit $\{T^\ell(o) \mid \ell \in \mathbb{Z}\}.$ Specifically, we have that for $\ell \in \mathbb{Z}^+,$
\begin{itemize}
    \item \(T^\ell(o)=(r_1 \circ \dots \circ r_\ell)(o)\),
    \item \(T^{-\ell}(o)=(r_{-1}\circ \dots \circ r_{-\ell})(o)\),
\end{itemize}
in the sense that the relative position of $P$ to $T^\ell(o)$ is the same as the relative position of $P_\ell$ to $o$. Here, $r_i$ denotes Euclidean reflection about $A_i.$

Moreover, the orbit $\{T^\ell(o) \mid \ell \in \mathbb{Z}\}$ is bounded if and only if the necklace is bounded, and periodic if and only if the necklace is periodic. 
\end{theorem}

Now we introduce Gutkin's \textit{cone and ray} construction. 
\begin{definition}
    Fix a convex $n$-gon $P$ in the plane. Since $P$ is an $n$-gon, we can draw $n$ straight lines $\ell_1, \dots, \ell_n$ through $o$ that are parallel to each side of $P$ (note that if $P$ has parallel sides, these lines may overlap; we assume from now on that the $n$ side directions of $P$ are distinct, so that there are exactly $2n$ cones---otherwise, $n$ should be replaced by the number of distinct side directions throughout, as in \cite{Gutkin}). The lines partition the plane into $2n$ cones, possibly degenerate, which we denote $C_1, \dots, C_{2n},$ in counter-clockwise order. 
    
    The lines $\ell_1, \dots, \ell_n$ create $2n$ rays in the plane, labeled as follows. Let cone $C_i$ be bounded by rays $R_i$ and $R_{i+1}$, and so forth, with the rays being labeled counter-clockwise. Note that we have the identities $C_{2n+1} = C_1, R_{2n+1} = R_1$ and $C_{n+i} = -C_i, R_{n+i} =-R_i.$ See Figure~\ref{fig: cone}.
\end{definition}

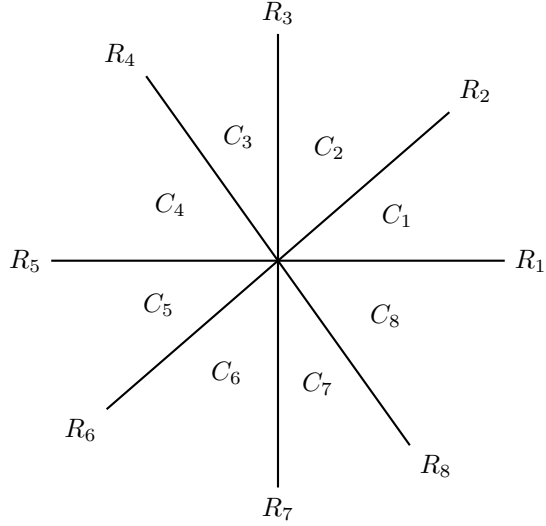
\begin{figure}[!htb]
    \centering
\begin{tikzpicture}[scale=1]
  \pgfmathsetmacro{\L}{3}

  \pgfmathsetmacro{\lenA}{sqrt(1*1 + 0*0)}              
  \pgfmathsetmacro{\lenB}{sqrt(0*0 + 2*2)}              
  \pgfmathsetmacro{\lenC}{sqrt(1.5*1.5 + 1.3*1.3)}      
  \pgfmathsetmacro{\lenD}{sqrt(0.5*0.5 + 0.7*0.7)}      
  \coordinate (R1end) at ({\L*1/\lenA}, {\L*0/\lenA});

  \coordinate (R2end) at ({\L*1.5/\lenC}, {\L*1.3/\lenC});

  \coordinate (R3end) at ({\L*0/\lenB}, {\L*2/\lenB});

  \coordinate (R4end) at ({\L*(-0.5)/\lenD}, {\L*0.7/\lenD});

  \coordinate (R5end) at ({\L*(-1)/\lenA}, {\L*0/\lenA});

  \coordinate (R6end) at ({\L*(-1.5)/\lenC}, {\L*(-1.3)/\lenC});

  \coordinate (R7end) at ({\L*0/\lenB}, {\L*(-2)/\lenB});

  \coordinate (R8end) at ({\L*0.5/\lenD}, {\L*(-0.7)/\lenD});

  \draw[thick] (0,0) -- (R1end) node[right] {$R_1$};
  \draw[thick] (0,0) -- (R2end) node[above right] {$R_2$};
  \draw[thick] (0,0) -- (R3end) node[above] {$R_3$};
  \draw[thick] (0,0) -- (R4end) node[above left] {$R_4$};
  \draw[thick] (0,0) -- (R5end) node[left] {$R_5$};
  \draw[thick] (0,0) -- (R6end) node[below left] {$R_6$};
  \draw[thick] (0,0) -- (R7end) node[below] {$R_7$};
  \draw[thick] (0,0) -- (R8end) node[below right] {$R_8$};

  \coordinate (C1pos) at ($0.3*(R1end) + 0.3*(R2end)$);
  \coordinate (C2pos) at ($0.3*(R2end) + 0.3*(R3end)$);
  \coordinate (C3pos) at ($0.3*(R3end) + 0.3*(R4end)$);
  \coordinate (C4pos) at ($0.3*(R4end) + 0.3*(R5end)$);
  \coordinate (C5pos) at ($0.3*(R5end) + 0.3*(R6end)$);
  \coordinate (C6pos) at ($0.3*(R6end) + 0.3*(R7end)$);
  \coordinate (C7pos) at ($0.3*(R7end) + 0.3*(R8end)$);
  \coordinate (C8pos) at ($0.3*(R8end) + 0.3*(R1end)$);

  \node at (C1pos) {$C_1$};
  \node at (C2pos) {$C_2$};
  \node at (C3pos) {$C_3$};
  \node at (C4pos) {$C_4$};
  \node at (C5pos) {$C_5$};
  \node at (C6pos) {$C_6$};
  \node at (C7pos) {$C_7$};
  \node at (C8pos) {$C_8$};

\end{tikzpicture}
    \caption{The cone and ray construction for $P_0$ in Figure~\ref{fig: necklace}.}
    \label{fig: cone}
\end{figure}

We say that a polygon on the plane is \textit{inside} the cone $C_j$ if it intersects $C_j$ and does not intersect the next cone $C_{j+1}$. Let $G$ be the group of reflections and translations of the plane. Then, define $\mathscr{P}$ to be the set of polygons \textit{congruent} to $P_0$ (i.e., polygons $Q$ such that $Q=g(P_0)$ for some $g\in G$) that are also strongly regular about $o.$ Denoting our original polygon by $P_0,$ the following lemma holds. 
\begin{lemma}[Gutkin \protect{\cite[pp. 436--437]{Gutkin}}]
Choose a cone $C_j.$ Then, consider the subset $\mathscr{P}_j\subset\mathscr{P}$ of polygons inside the cone $C_j.$ For each $Q\in \mathscr{P}_j,$ the head and the tail about $o$ are well-defined. Furthermore, the vector connecting the head and the tail does not depend on the choice of $Q\in \mathscr{P}_j.$
\end{lemma}

For a given cone $C_j$, we call the vector starting from the tail and ending at the head of any $Q\in \mathscr{P}_j$ the \textit{necklace vector in cone $C_j$}, and denote it as $\vec{a}_j.$ Note that $\vec{a}_{j+n} = -\vec{a}_j.$

Now, choose a point $A_1\in R_1.$ Draw the ray emanating from $A_1$ in the direction of $\vec{a}_1$, until it intersects $R_2$ at $A_2$. Then draw the ray in direction $\vec{a}_2$ from $A_2$, and repeat the process until we return to the ray $R_{2n+1} = R_1$ at $A_{2n+1}$. 

\begin{lemma}[Gutkin \protect{\cite[pp. 438--440]{Gutkin}}]
    The polygonal line generated by this process is closed; i.e., $A_{2n+1} = A_1.$
\end{lemma}

 This process traces a polygon called the \textit{necklace polygon}, denoted $Q$.
 
\begin{remark}
    Any changes in the definition of $Q$, i.e., the position of $o$, the choice of $R_1$, and the position of $A_1$ on $R_1$, changes $Q$ by translations and dilations only. Moreover, the necklace polygon is a convex, centrally symmetric $2n$-gon.
\end{remark}

Using this definition of the necklace polygon, we can now define \textit{quasi-rationality}.

\begin{definition}
    Take $Q = A_1\dots A_{2n+1}$ to be a necklace polygon of $P.$ Then there exist positive real numbers $r_1, r_2,\dots, r_{2n}$ such that for $1\leq i \leq 2n,$ we have 
    \[
    \overrightarrow{A_iA_{i+1}} = r_i\vec{a}_i.
    \]
    Note that $r_{n+i} = r_i.$ 
    
    We say that the polygon $P$ is \textit{quasi-rational} if $r_1, \dots, r_n $ are rational up to a common factor, i.e., $(r_1:r_2:\dots:r_n) \in \mathbb{QP}^{n-1}.$ We say that $P$ is \textit{rational} if the vertices of $P$ belong to a lattice.
\end{definition}

It should again be noted that every rational polygon is quasi-rational---refer to \cite{Gutkin} for a proof. We also collect a definition that is useful down the line.
\begin{definition}
    A \textit{truncated strip} is an infinite strip bordered by two parallel rays and a polygonal line. See Figure \ref{fig: truncatedstrip}.
\end{definition}

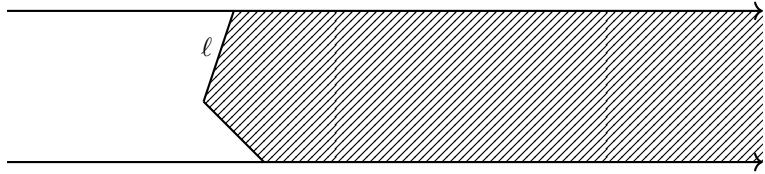
\begin{figure}[!htb]
    \centering
    \begin{tikzpicture}[scale=2]

  \coordinate (Tleft)  at (-1,1);   
  \coordinate (Tright) at (4,1);    
  \coordinate (Bleft)  at (-1,0);   
  \coordinate (Bright) at (4,0);    

  \coordinate (P1) at (0.7,0);        
  \coordinate (P2) at (0.3,0.4);      
  \coordinate (P3) at (0.5,1);      

  \draw[thick,->] (Tleft) -- (Tright);
  \draw[thick,->] (Bleft) -- (Bright);

  \draw[thick] (P1) -- (P2);
  \draw[thick] (P2) -- (P3);

  \fill[pattern=north east lines,pattern color=black]
    (P1) -- (P2) -- (P3) -- (Tright) -- (Bright) -- cycle;

  \node[left] at ($(P2)!0.6!(P3)$) {$\ell$};

\end{tikzpicture}
    \caption{A truncated strip.}
    \label{fig: truncatedstrip}
\end{figure}

\section{Orbits in Refractive Outer Billiards}\label{sec:gutkin}
We now extend to refractive outer billiards, where we now have \textit{refractive coefficients} $\lambda_1, \dots, \lambda_k$ that multiply to 1. Similarly to the regular outer billiard case, the points for which $T^\ell(o)$ is not well-defined for some $\ell\in \mathbb{Z}$ is a countable union of lines, and thus has zero measure. We will only work with \textit{strongly regular} points. 

The analog of the necklace construction is as follows. Fix a point $o$ at the origin and choose a polygon $P=P_0$ on the plane. Then, $P_1$ is the reflection of $P_0$ about its head $A_1,$ followed by a scaling by $1/\lambda_1$ about $A_1.$ More generally, $P_\ell$ is the reflection of $P_{\ell-1}$ about $A_\ell,$ followed by a scaling by $1/\lambda_\ell$ about $A_\ell.$ It is, of course, also the reflection of $P_{\ell+1}$ about its tail $A_\ell,$ followed by a scaling by $\lambda_{\ell+1}.$ 

Thus, given a polygon $P_0,$ we obtain the sequence $\{\dots, P_{-1}, P_0, P_1, \dots\},$ which we will call the \textit{refractive necklace of $P$ about $o$}. Moreover, we define the \textit{refractive necklace map} as the transformation that sends $P_\ell$ to $P_{\ell+1}.$ With this definition in mind, there is a natural correspondence between refractive billiards and the refractive necklace. See Figure~\ref{fig: refractivenecklace}.
\vspace{0.3cm}
\begin{figure}[!hbt]
    \centering
        \begin{tikzpicture}[scale=1.2]
  \coordinate (P0V1) at (0,0);
  \coordinate (P0V2) at (1,0);
  \coordinate (P0V3) at (1,2);
  \coordinate (P0V4) at (-0.5,0.7);

  \coordinate (A0) at (0,0);
  \coordinate (A1) at (1,2);
  \coordinate (A2) at (1,3);
  \coordinate (A3) at (1,5);
  \coordinate (O)  at (-4,2.2);

  \draw[thick]
    (P0V1) -- (P0V2) -- (P0V3) -- (P0V4) -- cycle;
  \draw[pattern=north east lines, pattern color=gray]
    (P0V1) -- (P0V2) -- (P0V3) -- (P0V4) -- cycle;

  \coordinate (P0mid) at (0.45,0.75);

  \coordinate (P1V1_full) at ($2*(A1) - (P0V1)$);
  \coordinate (P1V2_full) at ($2*(A1) - (P0V2)$);
  \coordinate (P1V3_full) at ($2*(A1) - (P0V3)$);
  \coordinate (P1V4_full) at ($2*(A1) - (P0V4)$);

  \coordinate (P1V1) at ($(A1)!0.5!(P1V1_full)$);
  \coordinate (P1V2) at ($(A1)!0.5!(P1V2_full)$);
  \coordinate (P1V3) at ($(A1)!0.5!(P1V3_full)$);
  \coordinate (P1V4) at ($(A1)!0.5!(P1V4_full)$);

  \draw[thick]
    (P1V1) -- (P1V2) -- (P1V3) -- (P1V4) -- cycle;
  \draw[pattern=north east lines, pattern color=gray]
    (P1V1) -- (P1V2) -- (P1V3) -- (P1V4) -- cycle;

  \coordinate (P1mid) at ($0.25*(P1V1)+0.25*(P1V2)+0.25*(P1V3)+0.25*(P1V4)$);

  \coordinate (R2) at (1,3); 

  \coordinate (P2V1) at ($(P0V1)+(0,3)$);
  \coordinate (P2V2) at ($(P0V2)+(0,3)$);
  \coordinate (P2V3) at ($(P0V3)+(0,3)$);
  \coordinate (P2V4) at ($(P0V4)+(0,3)$);

  \draw[thick]
    (P2V1) -- (P2V2) -- (P2V3) -- (P2V4) -- cycle;
  \draw[pattern=north east lines, pattern color=gray]
    (P2V1) -- (P2V2) -- (P2V3) -- (P2V4) -- cycle;

  \coordinate (P2mid) at ($0.25*(P2V1)+0.25*(P2V2)+0.25*(P2V3)+0.25*(P2V4)$);

  \coordinate (Pminus1V1_full) at (-1,0);
  \coordinate (Pminus1V2_full) at (0.5,-0.7);
  \coordinate (Pminus1V3_full) at (-1,-2);

  \coordinate (Pminus1V1) at ($(A0)!0.5!(Pminus1V1_full)$);
  \coordinate (Pminus1V2) at ($(A0)!0.5!(Pminus1V2_full)$);
  \coordinate (Pminus1V3) at ($(A0)!0.5!(Pminus1V3_full)$);

  \draw[thick]
    (A0) -- (Pminus1V1) -- (Pminus1V3) -- (Pminus1V2) -- cycle;
  \draw[pattern=north east lines, pattern color=gray]
    (A0) -- (Pminus1V1) -- (Pminus1V3) -- (Pminus1V2) -- cycle;

  \coordinate (Pminus1mid) at ($0.4*(Pminus1V3)+0.6*(A0)$);

  \node at (P0mid)       {$P_0$};
  \node at (P1mid)       {$P_1$};
  \node at (P2mid)       {$P_3$};
  \node at (Pminus1mid)  {$P_{-1}$};

  \fill (A0) circle (0.05)
        node[below right, xshift=10pt, yshift=3pt] {$A_0$};
  \fill (A1) circle (0.05) node[below right] {$A_1$};
  \fill (A2) circle (0.05) node[below left]  {$A_2$};
  \fill (A3) circle (0.05) node[below right] {$A_3$};

  \fill (Pminus1V3) circle (0.05) node[left] {$A_{-1}$};

  \fill (O) circle (0.05) node[left] {$o$};

  \draw[thick, shorten >=0pt, shorten <=-0.5cm] (A3)         -- (O);
  \draw[thick, shorten >=0pt, shorten <=-0.5cm] (A2)         -- (O);
  \draw[thick, shorten >=0pt, shorten <=-0.5cm] (A1)         -- (O);
  \draw[thick, shorten >=0pt, shorten <=-0.5cm] (A0)         -- (O);
  \draw[thick, shorten >=0pt, shorten <=-0.5cm] (Pminus1V3)  -- (O);

\end{tikzpicture}

\caption{Part of the refractive necklace with $\lambda_1 = 2, \lambda_2 = 1/2.$}
    \label{fig: refractivenecklace}
\end{figure}
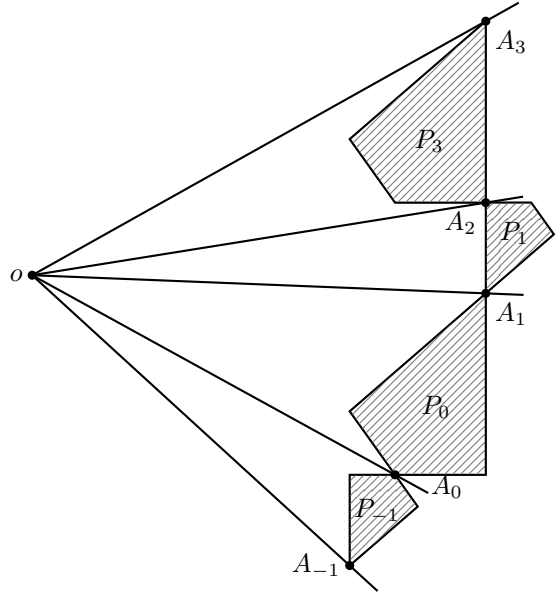
\begin{remark}
Because the $m$-th step of $T$ carries the index $\lambda_m$, a return $T^N(o) = o$ alone does not entail that the orbit repeats. The next step applies the index $\lambda_{N+1}$, which does not necessarily equal $\lambda_1$. We therefore say that the orbit is \emph{periodic} if the bi-infinite sequence $(T^m(o))_{m \in \mathbb{Z}}$ is periodic, i.e., there exists some $N>0$ such that $T^{m+N}(o) = T^m(o)$ and $\lambda_{m+N} = \lambda_m$ for all $m$. Periodicity of the necklace $(P_m)_{m \in \mathbb{Z}}$ is defined similarly.
\end{remark}

\begin{theorem}
The refractive necklace of $P$ about $o$ is defined simultaneously with the orbit $\{T^\ell(o) \mid \ell \in \mathbb{Z}\}$, and the two determine each other: for every $\ell \in \mathbb{Z}$ there is a direct similarity $\varphi_\ell$ of the plane satisfying
\[
\varphi_\ell(o) = T^\ell(o), \qquad \varphi_\ell(P_\ell) = P_0,
\]
with ratio $\mu_\ell = \lambda_1 \cdots \lambda_\ell$ for $\ell \ge 0$ and $\mu_\ell = (\lambda_0 \lambda_{-1} \cdots \lambda_{\ell+1})^{-1}$ for $\ell < 0$. In particular, the configuration $(T^\ell(o), P_0)$ is similar to $(o, P_\ell)$, with
\[
d\bigl(T^\ell(o), P_0\bigr) = \mu_\ell\, d(o, P_\ell).
\]
Moreover, the orbit is bounded if and only if the necklace is bounded, and periodic if and only if the necklace is periodic.
\end{theorem}

\begin{proof}
We write the operation of reflecting about $B$ and scaling by $\lambda$ about $B$ as $\rho_{B,\lambda}(z) := B+\lambda\big((2B-z)-B\big) = B + \lambda(B - z).$
First, note that $\rho_{B,\lambda}^{-1} = \rho_{B,1/\lambda}$. 

In this notation, one orbit iteration is $T^m(o) = \rho_{A_m, \lambda_m}\bigl(T^{m-1}(o)\bigr)$, where $A_m$ is the tangent point of the support line drawn from $T^{m-1}(o)$ to $P_0$. One necklace iteration is $P_m = \rho_{A_m, 1/\lambda_m}(P_{m-1})$, equivalently $P_{m-1} = \rho_{A_m, \lambda_m}(P_m)$.

Note that $\rho$ induces a direct similarity of the plane. Now, recall some facts about direct similarities. Let $S$ be an arbitrary direct similarity. First, $S$ preserves orientation, hence carries the head (and the tail) of a point $z$ about a polygon $Q$ to the head (tail) of $S(z)$ about $S(Q)$. Second, $S \circ \rho_{B,\lambda} \circ S^{-1} = \rho_{S(B),\lambda}$.

\textit{Correspondence.} Set $\varphi_0 := \operatorname{id}$ and, recursively,
\[
\varphi_\ell := \varphi_{\ell-1} \circ \rho_{A_\ell, \lambda_\ell} \quad (\ell \ge 1),
\qquad
\varphi_{-\ell} := \varphi_{-\ell+1} \circ \rho_{A_{-\ell+1}, \lambda_{-\ell+1}}^{-1} \quad (\ell \ge 1),
\]
so that $\varphi_\ell$ is a direct similarity of the stated ratio $\mu_\ell$, with linear part $(-1)^{\ell}\mu_\ell$ for $\ell \geq 0$. A straightforward induction argument on $\ell \ge 0$ proves that $\varphi_\ell(o) = T^\ell(o)$ and $\varphi_\ell(P_\ell) = P_0.$

Since each $\varphi_\ell$ carries heads to heads and tails to tails, the orbit and the necklace can be extended past position $\ell$ under exactly the same condition, so the two are defined simultaneously. Moreover, we have $$d(T^\ell(o), P_0) = d\bigl(\varphi_\ell(o), \varphi_\ell(P_\ell)\bigr) = \mu_\ell\, d(o, P_\ell).$$

\emph{Boundedness.} Since the product of the $\lambda_i$ over one full cycle equals $1$, the ratio $\mu_m$ depends only on $m \bmod k$ and takes finitely many positive values; note also that the scale of $P_m$ is $\mu_m^{-1}$, so $\operatorname{diam}(P_m)$ is uniformly bounded. By the displayed distance identity, $\{T^m(o)\}_{m \in \mathbb{Z}}$ is bounded iff $\{d(o, P_m)\}$ is bounded, iff $\bigcup_m P_m$ is bounded.

\emph{Periodicity.} Since we can always replace a period by a multiple of it, we may assume that $N$ is even and that $k \mid N$; then $\mu_N = 1$, so $\varphi_N$ has linear part $(-1)^N \mu_N = 1$ and is a translation.

Suppose the necklace is periodic with period $N$. Then $\varphi_N(P_N) = P_0 = P_N$, so $\varphi_N$ is the identity; hence $T^N(o) = \varphi_N(o) = o$. Since each step of $T$ depends only on the current point and the current index, and $\lambda_{m+N} = \lambda_m$, the orbit is periodic with period $N$.

Conversely, suppose the orbit is periodic with period $N$. Then $\varphi_N$ is a translation fixing $o$, hence the identity, and $P_N = \varphi_N(P_N) = P_0$, so the necklace is periodic with period $N$.
\end{proof}

Using this correspondence, we can study the refractive necklace instead of the refractive billiards map. First, we restrict the definition of $\mu$ earlier to $$\Lambda_i:= 1/(\lambda_1\lambda_2\cdots\lambda_i).$$ Along the necklace, we have at most $2k$ different possible configurations of polygons up to translation: a polygon is either oriented the same as $P$ or a reflection of $P$, and there are $k$ possible sizes, namely $P$ scaled by a factor of $\{\Lambda_i \mid 0\leq i \leq k-1\}.$  

Let $\mathcal{S}$ be the set of possible configurations of polygons on the plane through iterations of the necklace map. Define $\mathcal{S}_i^+$ as the set of polygons that are congruent to $P_i$ and have the same orientation as $P,$ and $\mathcal{S}_i^-$ as the polygons congruent to $P_i$ with opposite orientation. So, the polygons in $\mathcal{S}_i = \mathcal{S}_i^- \sqcup \mathcal{S}_i^+$ are congruent to $P$ scaled by a factor of ${\Lambda_i}$. Immediately, we have \[\mathcal{S}=\mathcal{S}_0^- \sqcup \mathcal{S}_0^+\sqcup \mathcal{S}_1^- \sqcup \dots\sqcup \mathcal{S}_{k-1}^- \sqcup \mathcal{S}_{k-1}^+.\] Returning to the cone and ray construction from the previous section, pick a ray $R_j.$ We will use $\pm$ to simultaneously define constructions that hold for both $+$ and $-$.  

\begin{definition}
    Let $\mathcal{S}_{R_j}$ denote the set of polygons $ Q \in \mathcal{S}$ that intersect $R_j$ but not $R_{j+1}.$ Geometrically, this represents the polygons in $\mathcal{S}$ that have ray $R_j$ as the ``furthest'' ray it touches. 
    
    Define $\mathcal{S}_{i,j}^+$ to be the intersection of $\mathcal{S}_i^+$ and $\mathcal{S}_{R_j},$ and define $\mathcal{S}_{i,j}^-$ similarly. This is the set of polygons in $\mathcal{S}$ that are congruent to $P_i,$ in the same orientation as $P$ (if $+$) or the opposite orientation (if $-$), and intersects ray $R_j$ but not $R_{j+1}.$ Note that $\mathcal{S}_i^+\sqcup \mathcal{S}_i^-$ contains all polygons congruent to $P_i$, but $\mathcal{S}_{i,j}^+ \sqcup \mathcal{S}_{i,j}^-$ only contains the polygons congruent to $P_i$ which intersect $R_j$ but not $R_{j+1}$. Thus, generally,
    \[
    \mathcal{S}_i^\pm \neq \bigsqcup_{j=1}^{2n} \mathcal{S}_{i,j}^\pm.
    \]
\end{definition}

Within each set $\mathcal{S}_{i,j}^\pm,$ each polygon is the same orientation and size, so they differ by a translation only. In other words, polygons $Q, Q' \in \mathcal{S}_{i,j}^+$ satisfy $Q = Q' + \vec{u}$ for some vector $\vec{u}.$ Then the heads of $Q$ and $Q'$ in $R_j$, $A_+$ and $A'_+$ respectively, satisfy $A_+ = A'_++ \vec{u}.$ 

Thus, there exists a bijective correspondence between the set $\mathcal{S}_{i,j}^\pm$ and its \textit{set of heads} $H_{i,j}^\pm$. The \textit{head function} $h_{i,j}^\pm:\mathcal{S}_{i,j}^\pm \to H_{i,j}^\pm$ that sends a polygon to its head satisfies the relation $h(Q+\vec{u}) = h(Q) + \vec{u}$ for all $Q\in \mathcal{S}_{i,j}^\pm$ and $\vec{u}$ such that $Q + \vec{u} \in \mathcal{S}_{i,j}^\pm.$

There exist unique vectors $\vec{d}_j\in R_j$ and $\vec{b}_j\in R_{j+1}$ such that $\vec{b}_j = \vec{d}_j + \vec{a}_j,$ where $\vec{a}_j$ is the head-tail vector inside the cone, and $\vec{b}_j, \vec{d}_j, \vec{a}_j$ form a triangle in the plane. 
\begin{lemma}
    The set $H_{i,j}^\pm$ is a truncated strip for each $1\leq i \leq k$ and $1\leq j \leq 2n$. In particular, $H_{i,j}^\pm$ is bounded by $R_j$, the ray $\tilde{R}_j \subseteq R_j + \Lambda_i \vec{b}_j$, and some polygonal line. Moreover, $H_{i,j}^\pm$ is $H_{0,j}^\pm$ scaled by $\Lambda_i.$ See Figure~\ref{fig: parallelogram}.
\end{lemma}
\begin{proof}
    The first part is a direct generalization of the proof in pp. 441--442 of~\cite{Gutkin}, so it is omitted: polygons in $S^\pm_{i,j}$ are the $\Lambda_i$-scaled copies of those in $S^\pm_{0,j}$, hence $H^\pm_{i,j} = \Lambda_i H^\pm_{0,j}$. For the second part, note that polygons in $\mathcal{S}_{i,j}^\pm$ are just polygons in $\mathcal{S}_{0,j}^\pm$ scaled by $\Lambda_i.$ By virtue of the construction of the truncated strip in~\cite{Gutkin}, it follows that $H_{i,j}^\pm$ is simply $H_{1,j}^\pm$ scaled by $\Lambda_i$.
\end{proof}
Note, of course, that the strips $\mathcal{S}_{i,j}^+$ and $\mathcal{S}_{i,j}^-$ need not be the same, since their orientations are different.

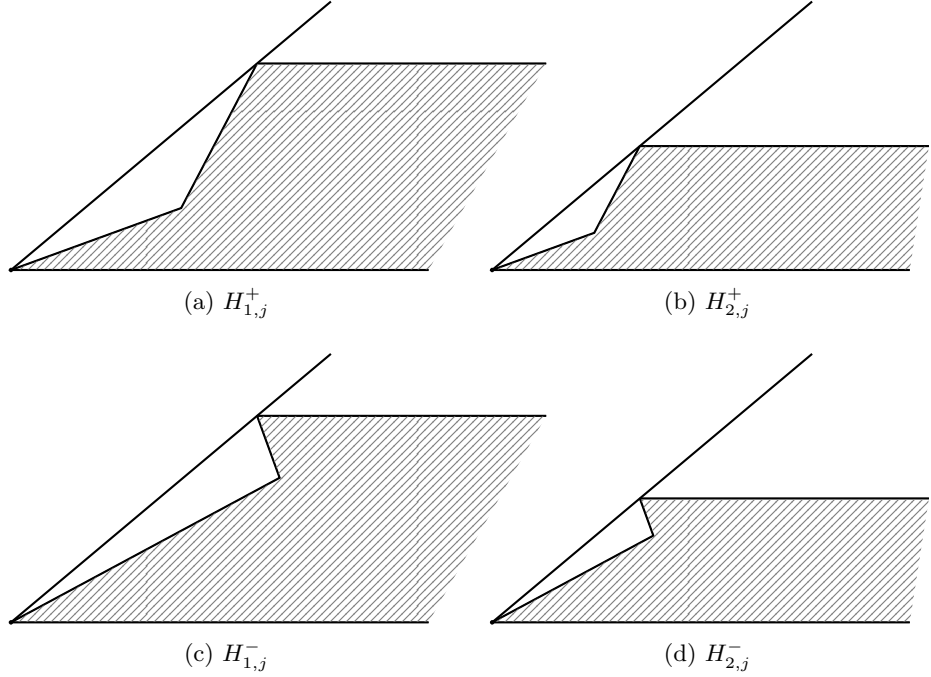
\begin{figure}[!hbt]
\centering
\pgfmathsetmacro{\L}{6.5}
\pgfmathsetmacro{\ang}{40}
\pgfmathsetmacro{\lamtwo}{2}

\pgfmathsetmacro{\LparTop}{4.5}
\pgfmathsetmacro{\LparBot}{4.5}
    \begin{subfigure}[b]{0.475\textwidth}
        \centering
        \begin{tikzpicture}[scale=0.85, transform shape]
  \begin{scope}[shift={(0,7)}]
    \fill (0,0) circle (1pt);
    \draw[thick] (0,0) -- (\L,0);
    \draw[thick] (0,0) -- (\ang:\L);
    \coordinate (Btop) at (\ang:5);
    \draw[thick] (Btop) -- ($(Btop)+(\LparTop,0)$);
    \draw[thick] (Btop) -- ($0.3*(Btop)+(1.5,0)$) -- (0,0);
  \draw[pattern=north east lines, pattern color=gray]
    ($(Btop)+(\LparTop,0)$) -- (Btop) -- ($0.3*(Btop)+(1.5,0)$) -- (0,0) -- (\L,0);
  \end{scope}
\end{tikzpicture}
\caption{$H_{1,j}^+$}
    \end{subfigure}
    \hfill
    \begin{subfigure}[b]{0.475\textwidth}
        \centering
        \begin{tikzpicture}[scale=0.85, transform shape]
  \begin{scope}[shift={(0,7)}]
    \fill (0,0) circle (1pt);
    \draw[thick] (0,0) -- (\L,0);
    \draw[thick] (0,0) -- (\ang:\L);
    \coordinate (Btop) at (\ang:3);
    \draw[thick] (Btop) -- ($(Btop)+(\LparTop,0)$);
    \draw[thick] (Btop) -- ($0.3*(Btop)+(0.9,0)$) -- (0,0);
  \draw[pattern=north east lines, pattern color=gray]
    ($(Btop)+(\LparTop,0)$) -- (Btop) -- ($0.3*(Btop)+(0.9,0)$) -- (0,0) -- (\L,0);
  \end{scope}
\end{tikzpicture}
\caption{$H_{2,j}^+$}
    \end{subfigure}
    \vskip\baselineskip
    \begin{subfigure}[b]{0.475\textwidth}
        \centering
        \begin{tikzpicture}[scale=0.85, transform shape]
  \begin{scope}[shift={(0,7)}]
    \fill (0,0) circle (1pt);
    \draw[thick] (0,0) -- (\L,0);
    \draw[thick] (0,0) -- (\ang:\L);

    \coordinate (Btop) at (\ang:5);
    \draw[thick] (Btop) -- ($(Btop)+(\LparTop,0)$);
    \draw[thick] (Btop) -- ($0.7*(Btop)+(1.5,0)$) -- (0,0);
  \draw[pattern=north east lines, pattern color=gray]
    ($(Btop)+(\LparTop,0)$) -- (Btop) -- ($0.7*(Btop)+(1.5,0)$) -- (0,0) -- (\L,0);
  \end{scope}
\end{tikzpicture}
\caption{$H_{1,j}^-$}
    \end{subfigure}
    \hfill
    \begin{subfigure}[b]{0.475\textwidth}
        \centering
\begin{tikzpicture}[scale=0.85, transform shape]
  \begin{scope}[shift={(0,7)}]
    \fill (0,0) circle (1pt);
    \draw[thick] (0,0) -- (\L,0);
    \draw[thick] (0,0) -- (\ang:\L);
    \coordinate (Btop) at (\ang:3);
    \draw[thick] (Btop) -- ($(Btop)+(\LparTop,0)$);
    \draw[thick] (Btop) -- ($0.7*(Btop)+(0.9,0)$) -- (0,0);
  \draw[pattern=north east lines, pattern color=gray]
    ($(Btop)+(\LparTop,0)$) -- (Btop) -- ($0.7*(Btop)+(0.9,0)$) -- (0,0) -- (\L,0);
  \end{scope}
\end{tikzpicture}
\caption{$H_{2,j}^-$}
    \end{subfigure}
\caption{Four example regions with $\lambda_2 = 3/5.$}
\label{fig: parallelogram}
\end{figure}

For each cone $C_j$, we have $2k$ truncated strips. We take the \textit{disjoint union of truncated strips} \[H_j:=\bigsqcup_{i=1}^k (H_{i,j}^+ \sqcup H_{i,j}^-).\] By the bijective correspondence between $\mathcal{S}_{i,j}^\pm$ and $H_{i,j}^\pm$, each point in $H_j$ corresponds to a unique polygon.

Using this new language, we can define the \textit{next cone map.}
\begin{definition}
    Take a polygon $Q \in \mathcal{S}_{i,j}^\pm.$ Continue applying the refractive necklace map to $Q$ until it intersects ray $R_{j+1}.$ Denote this new polygon as $Q'.$ Note that this process terminates since the vector $\vec{a}_j$ is not parallel to either of the rays $R_j$ or $R_{j+1}.$
    
    The \textit{next cone map in cone $C_j$}, denoted $f_j$, maps heads to heads: $f_j: H_j\to H_{j+1}.$ It maps the head of $Q$ in $H_j$ to the head \textit{in cone $C_{j+1}$} of $Q',$ which lies in $H_{j+1}.$ Importantly, we need to add a head correction vector $\vec{h}^{\,\varepsilon,\tilde{\varepsilon}}_{p,j}$ that accounts for a change in the head vertex; it depends on the pair of orientations $(\varepsilon, \tilde{\varepsilon})$, on the scale $\Lambda_p$, and on the cone $C_j$.
\end{definition}

\begin{remark}
    Taking $H$ to be $H_1 \sqcup H_2 \sqcup \dots \sqcup H_{2n},$ the maps $f_j$, when combined, naturally extend to the \textit{general} next cone map $f:H\to H$. Note that $f$ satisfies \( f(H_j) \subseteq H_{j+1}.\)
\end{remark}
The map $f_j$ can also be viewed as a map of tuples $(A, i, \varepsilon) \mapsto (\tilde{A}, \tilde{i}, \tilde{\varepsilon}).$ Here, $A$ represents the location of the polygon's head in $C_j$; the values $i, \varepsilon$ represent the truncated strip $H_{i,j}^\varepsilon$ to which the polygon belongs. Similarly, $\tilde{A}$ is the location of the head in cone $C_{j+1},$ and the polygon lies in $H_{\tilde{i},j+1}^{\tilde{\varepsilon}}.$ By our prior discussion, we know that $\tilde{A}$ is $A + C \cdot \vec{a}_j + \vec{h},$ where $C$ is some constant and $\vec{h}$ is some head correction vector. 

To find the precise values of $C$ and $\vec{h},$ we partition each truncated strip $H_{i,j}^\pm$ into regions $\pi_{i, j, m}^\pm$ for $m\in \mathbb{Z}^+.$ The region $\pi_{i, j, m}^\pm$ is the set of heads in $H_{i,j}^\pm$ such that $f_j$ is equal to $m$ iterations of the refractive necklace map. In the following proposition, we describe these regions.

\begin{proposition}
\label{prop: partition}
On ray $R_j,$ take the infinite sequence of points that have distance $\Lambda_{i+1}\vec{d_j},\, \Lambda_{i+2}\vec{d_j},\, \dots$ away from each other, starting at the origin. Since the indices are viewed modulo $k$, the sequence of gaps has period $k$. Now, draw a line parallel to $R_{j+1}$ through each point. Taking intersections of these lines with the ray $R_j + \Lambda_i\vec{b}_j,$ we obtain a periodic sequence of parallelograms. 

The region $\pi_{i,j,m}^\pm$ is the intersection of the $m$-th parallelogram with the truncated strip $H_{i,j}^\pm$. Note that $\pi_1$ may not be a parallelogram since $H_{i,j}^\varepsilon$ is a truncated strip.
\end{proposition}
\begin{proof}
    This follows from the fact that the first iteration of a head in $H_{i, j}^\pm$ adds $\Lambda_{i+1}\vec{a_j}$ to $P,$ the second iteration adds $\Lambda_{i+2}\vec{a_j},$ and so on. Partitioning the truncated strip via lines parallel to $R_{j+1}$ through these points, we get the construction above. 
\end{proof}

Indeed, using this partition, we can completely describe the next cone map $f_j.$ 

\begin{lemma}\label{lem: nextCone}
    Let $f_j: H_j \to H_{j+1}$ be the next cone map on cone $C_j$. Using the tuple notation above, $f_j$ maps $(A, i, \varepsilon)$ to $(\tilde{A}, \tilde{i}, \tilde{\varepsilon}).$ If $A$ lies in $\pi_{i, j, m}^\varepsilon$, we have the following:
    \begin{equation}
        \begin{cases}
            \tilde{A} = A + \big(\Lambda_{i+1} + \Lambda_{i+2} + \dots + \Lambda_{i+m} \big)\cdot \vec{a} + \vec{h}_i^{\tilde{\varepsilon}}, \\
            \tilde{\varepsilon} = (-1)^m\varepsilon, \\
            \tilde{i} \equiv i+m \bmod k.
        \end{cases}
    \end{equation}
\end{lemma}
\begin{proof}
    Since our polygon has scale $\Lambda_i$, the first iteration of the refractive necklace map gives a polygon of scale $\Lambda_{i+1}$. Repeating $m$ times, we get a total of $m$ polygons with scales $\Lambda_{i+1}, \dots, \Lambda_{i+m}.$ Multiply this constant by $\vec{a},$ then add the head correction vector $\vec{h}_i^{\tilde{\varepsilon}}$ to get the first part. The other two parts follow immediately from the definition of $\pi_k$.
\end{proof}

\begin{corollary}
Recall the definition of $\vec{b}, \vec{d}$ from Proposition~\ref{prop: partition}. Writing $\Lambda = 1 + \Lambda_1 + \dots + \Lambda_{k-1},$
the following relation holds: 
\[
f_j \Big(A + 2\Lambda \vec{d_j}\;\Big) = f_j(A) + 2\Lambda\vec{b_j}.
\]
\end{corollary}
\begin{proof}
    Suppose $A$ lies in $\pi_{i, j,m}^\varepsilon$. By construction of $\pi$, it follows that $A+2\Lambda\vec{d_j}$ lies in $\pi_{m+2k},$ so by Lemma~\ref{lem: nextCone} the head positions agree. Note that the head correction vector is also identical. 
    Moreover, $(-1)^{m+2k}\varepsilon  = (-1)^m \varepsilon, $ and $i+m \equiv i+m+2k \pmod k.$ 
\end{proof}

Now, we introduce the \textit{first return map} $F$. 

\begin{definition}
    Consider a polygon with head in $H_1$. The \textit{first return map} $F= f_{2n}\circ \dots \circ f_1$ represents the head of the first polygon that intersects $R_1$ again after completing a ``full loop.'' In the same sense as the next cone map $f,$ the function $F:H_1 \to H_{2n+1} = H_1$ maps a tuple $(A,i, \varepsilon)$ to $(\tilde{A}, \tilde{i}, \tilde{\varepsilon}).$ 
\end{definition}

With the first return map defined, we can finally start analyzing the structure of orbits. Note that the first return map is invertible. 
\begin{proposition}\label{firstReturnPeriodic}
    Let the polygon $P$ be quasi-rational. Then there exists $N\in\mathbb{Z}^+$ such that 
    \[F(x+2N\Lambda\vec{d}_1) = F(x) + 2N\Lambda\vec{d}_1. \]
\end{proposition}
\begin{proof}

    We have $oA_{j+1} =r_j\vec{b_j} = r_{j+1}\vec{d}_{j+1}.$ If the polygon $P$ is quasi-rational, then we can assume that \[r_j = \Lambda n_j, n_j\in \mathbb{N}\; (1\leq j \leq 2n).\]
    Now, applying the identity 
    \[
    f_j\Big(A + 2\vec{d}_j\Lambda\Big) = f(A) + 2\vec{b}_j\Lambda,
    \] we get
    \[
    f_j(x+2n_j\Lambda \vec{d}_j) = f_j(x) + 2n_j\Lambda \vec{b}_j = f_j(x) + 2n_{j+1}\Lambda \vec{d}_{j+1}.
    \]
    Iterating for $j=1,\dots, 2n,$ we get 
    \[
    (f_{2n}\circ \dots \circ f_2\circ f_1)(x+2n_1\Lambda \vec{d}_1)= (f_{2n}\circ \dots \circ f_2\circ f_1)(x) + 2n_1\Lambda \vec{d}_1,
    \] so our proof is complete: set $N=n_1.$
\end{proof}

\begin{corollary}
    Let $\vec{p}:=2n_1\Lambda \vec{d}_1,$ and choose $k\in \mathbb{Z}$, $x\in H_1.$ If the point $x+k\vec{p}$ lies in $H_1,$ then 
    \[
    F(x+k\vec{p}) = F(x) + k\vec{p}.
    \]
\end{corollary}

Using the corollary above, we can reduce the entire first return map to a \textit{fundamental domain} $\Pi,$ which is the disjoint union of all $\pi_{i,1,m}^\pm$ with bottom parts inside some choice of $\vec{p}$ in $R_1.$

Thus, we can redefine $F$ as follows:
\begin{definition}
    Let $\Phi:\Pi\to\Pi$ and $\tau:\Pi\to\mathbb{Z}$ be functions such that $F(x) = \Phi(x) + \tau(x)\vec{p}.$ 
\end{definition}

We have the following three results of Gutkin that we can apply. 

\begin{lemma}[Gutkin \protect{\cite[p. 445]{Gutkin}}]
    The pair $(\Phi, \tau)$ uniquely determines the first return map $F$. The mapping $\Phi$ is invertible and $\Phi, \Phi^{-1}: \Pi \rightarrow \Pi$ are local translations.
\end{lemma}
\begin{corollary}[Gutkin \protect{\cite[p. 446]{Gutkin}}]
   The first return map $F$ is uniquely determined by the pair of maps $(\Phi: \Pi \to \Pi, \tau: \Pi \to\mathbb{Z}_{\geq -1})$. The function $\tau$ corresponding to an invertible mapping $F$ can take values $-1,0,1$ only.
\end{corollary}
\begin{theorem}[Gutkin \protect{\cite[p. 446]{Gutkin}}]\label{thm: bddperiodic}
    If $F$ is invertible and satisfies the condition in Proposition \ref{firstReturnPeriodic}, then the orbits $\left\{F^k(x):-\infty<k<\infty\right\}$ are bounded. If $F$ is invertible and the translation vectors $\vec{t}(\varepsilon, i, k)$ defining $F$ generate a discrete group, then the orbits of $F$ are periodic.
\end{theorem}
Thus we obtain the boundedness theorem for the refractive outer billiards system:
\begin{theorem}\label{thm: bdd}
    Let $T$ be the refractive outer billiard mapping about a polygon $P$. If $P$ is quasi-rational
then the orbits of $T$ are bounded.
\end{theorem}
\begin{proof}
    We know that the orbit is bounded if and only if the necklace is bounded. If the first return map is bounded, the infinite necklace is bounded. 
\end{proof}
Rationality is slightly different:
\begin{theorem}\label{thm: periodic}
    If $P$ is rational and each of $\lambda_1, \dots, \lambda_k$ are rational, then the orbits of the refractive outer billiards map $T$ are periodic. 
\end{theorem}
\begin{proof}
    First, note that since $P$ is rational, it is also quasi-rational. By Theorem~\ref{thm: bdd}, we know that the orbit is bounded. 
    If $P$ is rational and each of $\lambda_1, \dots, \lambda_k$ are rational, then after clearing a common denominator, all heads along one orbit lie in a single fixed lattice. Since a bounded subset of a lattice is finite, the sequence of first-return heads is eventually periodic. The invertibility of $F$ finishes the proof. 
\end{proof}

\section{The Existence of Periodic Orbits}\label{sec:culter}
Now, we generalize Culter's theorem: every refractive outer billiard has a periodic orbit. It is worth noting that the proof for this in~\cite{Culter} avoids the usage of the necklace map technique. We present a slightly modified approach here involving the use of necklaces. 

The main idea of the proof is as follows. Instead of directly searching for a periodic orbit of $T,$ we will look for a periodic orbit of $T^{2k}.$ After the $2k$-th power of the refractive necklace map, the polygon returns to its original shape. In other words, $T^{2k}(P)$ is a translation of $P$. 

We will look for a necklace that satisfies the following conditions:
\begin{enumerate}
    \item Each ray contains exactly one polygon from the refractive necklace, which are all translations of each other (i.e., they can be obtained from each other by applying $T^{2k}$ multiple times). Denote the polygon on ray $R_j$ as $Q_j$.
    \item 
    If the first step is met, then the following identity holds:
    \[
    Q_{j+1} = Q_j + 2\Lambda\overrightarrow{ a_{j}} \cdot p_j,
    \] where $p_j$ is some positive integer satisfying $p_{n+i} = p_i.$
    \item The \textit{head} of $Q_j$ lies in the parallelogram $\pi_{0, j, 2kp_j}^+.$
\end{enumerate}
The last two conditions ensure that $Q_{j+1} = T^{2k\cdot p_j}(Q_j).$ If these conditions are met, we can simply fill in the remaining refractive necklace to obtain a periodic orbit. 

First, we introduce a lemma that helps us simplify the head condition. 

\begin{lemma}
\label{lem: interior}
    Consider a convex $n$-gon $P$ and its $2n$ cones. Suppose $P$ lies in a cone bordered by rays $R$ and $R'$. Construct the parallelogram $\pi$ that has sides parallel to $R$ and $R'$ and $\overrightarrow{A_-A_+}$ as a diagonal. Then $P$ is completely contained in $\pi.$
\end{lemma}
\begin{figure}[!htb]
    \centering
\usetikzlibrary{calc}

\begin{tikzpicture}[scale=0.8]
  \coordinate (O) at (0,0);
  \coordinate (Rend)  at (30:7);
  \coordinate (Rprime) at (90:7);
  \draw[thick] (O) -- (Rend);
  \draw[thick] (O) -- (Rprime);
  \node[above right] at (Rend) {$R$};
  \node[above] at (Rprime) {$R'$};
  
  \coordinate (P1) at (40:5);
  \coordinate (P2) at (50:5.95);
  \coordinate (P3) at (65:6);
  \coordinate (P4) at (80:4.5);
  \coordinate (P5) at (55:3);

  \fill[gray!20](P1) -- (P2) -- (P3) -- (P4) -- (P5) -- cycle;
  \draw[thick, pattern=north east lines, pattern color=gray]
    (P1) -- (P2) -- (P3) -- (P4) -- (P5) -- cycle;
  \coordinate (Aminus) at (40:7);  
  \coordinate (Aplus)  at (80:7);  

  \draw[thick] (0,0) -- (Aminus);
  \draw[thick] (0,0) -- (Aplus);
  \fill (P1) circle (1.5pt) node[right] {$A_-$};
  \fill (P4)  circle (1.5pt) node[left] {$A_+$};
  \fill (0,0) circle (1.5pt) node[below left] {$o$};
  \path let \p1 = (P4) in
    coordinate (B) at (\x1, {\x1*tan(30)});
  \coordinate (P4up) at ($(P4)+(0,3)$);
  \draw[thick] (B) -- (P4up);
  \node[right] at (P4up) {$\ell$};
  \fill[gray!20] (0,0) -- (P4) -- (B) -- cycle;
  \fill[thick, pattern=north east lines, pattern color=gray] (0,0) -- (P4) -- (B) -- cycle;
  \draw[thick] (0,0) -- (P4) -- (B) -- cycle; 

  \fill (B) circle (1.5pt) node[below] {$B$};

\end{tikzpicture}
    \caption{A diagram for the proof of Lemma~\ref{lem: interior}.}
    \label{fig: interior}
\end{figure}
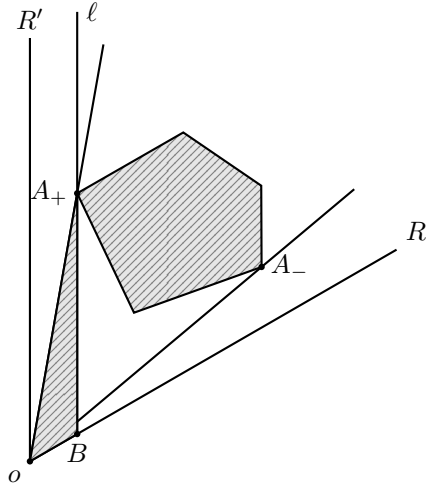
\begin{proof}
    See Figure~\ref{fig: interior}.
    Draw $\ell,$ the line that passes through $A_+$ and is parallel to $R'.$ We claim that every vertex of $P$ lies on the opposite side of $\ell$ as $R'.$ 

    Let $B$ be the intersection of $\ell$ and $R.$ Since $A_+$ is the head, no vertex of $P$ can lie in the truncated strip formed by $R', \overline{OA_+},$ and $\ell.$ Suppose for the sake of contradiction that some vertex $v$ lies in triangle $OA_+B.$ 
    
    Since $P$ is convex, there exists a vertex adjacent to $A_+$, which we denote $v',$ that lies in $\triangle OA_+B.$ In this case, the ray from $o$ in the direction of $v'A_+$ lies strictly between $R$ and $R',$ so that $C$ is not a cone: the ray from $o$ in the direction of $v'A_+$ lies between $R'$ and $oA_+,$ which gives a contradiction. Repeating a similar argument with $A_-$ gives that every vertex of $P$ lies inside $\pi,$ and by convexity, $P$ is contained in $\pi.$
\end{proof}

Now, pick an arbitrary point $x$ in the interior of $P.$ Using the lemma above, we have the following fact.

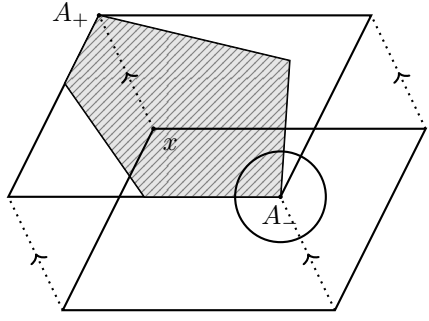
\begin{figure}[!htb]
    \centering
    \usetikzlibrary{calc}
    \begin{tikzpicture}[scale = 0.6]
        \coordinate (O) at (0,0);
        \coordinate (A) at (6,0);
        \coordinate (B) at (8,4);
        \coordinate (C) at (2,4);

        \draw[thick]
        (O) -- (A) -- (B) -- (C) -- cycle;

        \coordinate (P1) at (3,0);
        \coordinate (P2) at (1.25, 2.5);
        \coordinate (P) at (6.2,3);

        \draw[thick]
        (C) -- (P2) -- (P1) -- (A) -- (P) -- cycle;        \fill[gray!20] (C) -- (P2) -- (P1) -- (A) -- (P) -- cycle;
        \draw[pattern=north east lines, pattern color=gray]
        (C) -- (P2) -- (P1) -- (A) -- (P) -- cycle;   
        \coordinate (x) at (3.2,1.5);
        \coordinate (xo) at ($(O) + (x) - (C)$);
        \coordinate (xa) at ($(A)+ (x) - (C)$);
        \coordinate (xb) at ($(B) + (x) - (C)$);
        \coordinate (xc) at ($(C) + (x) - (C)$);

        \draw[thick]
        (xo) -- (xa) -- (xb) -- (xc) -- cycle;
        \begin{scope}[thick,decoration={
            markings,
            mark=at position 0.5 with {\arrow{>}}}
            ] 
            \draw[thick, dotted, postaction={decorate}](xo) -- (O);
        \draw[thick, dotted, postaction={decorate}](xa) -- (A); 
        \draw[thick, dotted, postaction={decorate}](xb) -- (B); 
        \draw[thick, dotted, postaction={decorate}](xc) -- (C); 
        \end{scope}
    \fill (x) circle (0.05) node[below right] {$x$};
    \fill (C) circle (0.05) node[left] {$A_+$};
    \fill (A) circle (0.05) node[below] {$A_-$};

    \draw[thick] (A) circle (1);
    \end{tikzpicture}
    \caption{An $\varepsilon$-ball for $x$ centered at $A_-$. Each dotted vector is $\vec{v}.$}
    \label{fig: ball}
\end{figure}

\begin{proposition}
    For each ray $R_j,$ there exists $\varepsilon_j>0$ such that if $P$ is translated so that $x$ is in one of the $\varepsilon_j$-balls centered at $2\Lambda p_j \vec{d}_j,$ where $p_j \in \mathbb{Z}^+$, then the head of the translated polygon lies inside the region $\pi_{0, j, 2kp_j}^+$. 
\end{proposition}
\begin{proof}
    First note that the ``bottom right'' corner of $\pi_{0, j, 2kp_j}^+$ is placed at $2\Lambda \vec{r}_jp_j$ along $R_j$ from $o$. Therefore, we can prove this for general translations of $\pi$. Let $\vec{v}$ be the vector difference of $x$ and $A_+$. Since $\vec{v}$ is completely contained in the parallelogram, the points $q$ for which $q+ \vec{v}$ is inside the parallelogram is a translation of $\pi$, which contains the bottom right corner. It follows that such an $\varepsilon_j$ exists. See Figure~\ref{fig: ball}.
\end{proof}

\begin{figure}[!htb]
    \centering
    \begin{tikzpicture}[scale=1.2]
  \coordinate (O) at (0,0);
  \draw[thick,->] (O) -- (7,0);
  \node[below left] at (O) {$o$};
  \node[above] at (7,0) {$R_j$};
  \coordinate (P1) at (3,0);
  \coordinate (P2) at (6,0);
  \draw (P1) circle [radius=0.1];
  \draw (P2) circle [radius=0.1];
  \pgfmathsetmacro{\pw}{0.8}  
  \pgfmathsetmacro{\ph}{0.4}  
  \pgfmathsetmacro{\sx}{0.3}  
  \coordinate (P1BR) at (P1);                         
  \coordinate (P1BL) at ($(P1BR)+(-\pw,0)$);          
  \coordinate (P1TL) at ($(P1BL)+(\sx,\ph)$);         
  \coordinate (P1TR) at ($(P1BR)+(\sx,\ph)$);         

  \draw[fill=gray!20]
    (P1BR) -- (P1BL) -- (P1TL) -- (P1TR) -- cycle;
  \coordinate (P2BR) at (P2);                         
  \coordinate (P2BL) at ($(P2BR)+(-\pw,0)$);          
  \coordinate (P2TL) at ($(P2BL)+(\sx,\ph)$);         
  \coordinate (P2TR) at ($(P2BR)+(\sx,\ph)$);         

  \draw[fill=gray!20]
    (P2BR) -- (P2BL) -- (P2TL) -- (P2TR) -- cycle;
  \node[below] at ($(O)!0.5!(P1)$) {$2\Lambda r_j$};
  \node[below] at ($(P1)!0.5!(P2)$) {$2\Lambda r_j$};
  \fill (P1) circle (2.5pt) node[below] {$\varepsilon_j$};
  \fill (P2) circle (2.5pt) node[below] {$\varepsilon_j$};
\end{tikzpicture}

    \caption{The $\varepsilon_j$-balls along $R_j$. The first parallelogram is $\pi_{1, j, 2k}^+,$} and the second is $\pi_{1, j, 4k}^+.$
    \label{fig:placeholder}
\end{figure}
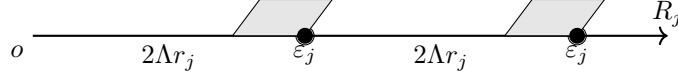

Taking $\min(\varepsilon_1, \varepsilon_2, \dots, \varepsilon_{2n}) = \varepsilon,$ we know that if on the ray $R_j$ the point $x$ lies in a $\varepsilon$-ball centered at distance $2\Lambda r_j p_j$ for any $p_j\in \mathbb{Z}^+,$ the refractive necklace map will send the polygon to the next cone in $2kp_j$ iterations.

At this point, it suffices to find a polygon with vertices in these $\varepsilon$-balls such that the vector difference between points on adjacent rays $R_j$ and $R_{j+1}$ is an integer multiple of $2\Lambda\vec{a}_j.$ Then, we can place $x$ at each vertex of this new polygon to obtain a periodic orbit.

\begin{lemma}[Tabachnikov \protect{\cite[pp. 4--5]{Culter}}]
    For any $\delta>0,$ there exists arbitrarily large $q\in \mathbb{R}$ and positive integers $p_1, p_2, \dots, p_{2n}$ such that for each $j,$
\[
\lvert  qr_j -p_j \rvert < \delta.
\] 
\end{lemma}

\begin{remark}
    In other words, we can approximate $(r_1: r_2: \dots : r_{n}) \in \mathbb{RP}^{n-1}$ arbitrarily well with $(p_1 : \dots : p_n) \in \mathbb{ZP}^{n-1}.$ Note that since $r_{n+j} = r_j,$ we also have $p_{n+j} = p_j.$
\end{remark}

\begin{theorem}
For any choice of $P$ and $\lambda_1, \dots, \lambda_k$, there exists a periodic orbit. Moreover, the set of points that form a periodic orbit has positive measure in the plane.
\end{theorem}

\begin{proof}
Let $\varepsilon := \min(\varepsilon_1, \dots, \varepsilon_{2n})$, with $\varepsilon_j$ as in Proposition 17. We seek a more specific polygon with side length vectors \[2\Lambda \vec{a}_1p_1, \dots, 2\Lambda\vec{a}_np_n, -2\Lambda \vec{a}_1p_1, \dots, -2\Lambda\vec{a}_np_n,\] whose $j$-th vertex lies within $\varepsilon$ of the point $2\Lambda p_j\vec{d}_j$ on $R_j$. Note that the sum of these is the zero vector, so that we actually return to the original point. 

For $p_1, \dots, p_{2n}$ with $p_{n+j} = p_j$, define the vertices
\[
V_1 := 2\Lambda p_1\vec{d}_1, \qquad V_{j+1} := V_j + 2\Lambda p_j\vec{a}_j,
\]
and the deviations $X_j := V_j - 2\Lambda p_j\vec{d}_j$, so that $X_1 = 0$. Since $\vec{a}_{n+j} = -\vec{a}_j$ and $p_{n+j} = p_j$, the path closes. From $\vec{b}_j = \vec{d}_j + \vec{a}_j$ and $r_j\vec{b}_j = r_{j+1}\vec{d}_{j+1}$ we get $\vec{d}_{j+1} = \frac{r_j}{r_{j+1}}\vec{b}_j$, and therefore
\[
X_{j+1} = X_j + 2\Lambda\bigl(p_j\vec{d}_j + p_j\vec{a}_j - p_{j+1}\vec{d}_{j+1}\bigr)
        = X_j + 2\Lambda\Bigl(p_j - \tfrac{r_j}{r_{j+1}}\,p_{j+1}\Bigr)\vec{b}_j.
\]
Choose $\delta > 0$ small enough that
\[
2\Lambda\,\delta\,\Bigl(1 + \max_j \tfrac{r_j}{r_{j+1}}\Bigr)\,\max_j\|\vec{b}_j\| \;<\; \frac{\varepsilon}{2n},
\]
and, by Lemma 18, choose $q$ and $p_1, \dots, p_{2n} \in \mathbb{Z}^+$ with $|qr_j - p_j| < \delta$ for all $j$. Then,
\[
\Bigl|p_j - \tfrac{r_j}{r_{j+1}}\,p_{j+1}\Bigr|
\;\le\; |p_j - qr_j| + \tfrac{r_j}{r_{j+1}}\,|qr_{j+1} - p_{j+1}|
\;<\; \delta\Bigl(1 + \tfrac{r_j}{r_{j+1}}\Bigr),
\]
so $\|X_{j+1}\| < \|X_j\| + \varepsilon/2n$, and by induction
\[
\|X_j\| \;\le\; \frac{(j-1)\,\varepsilon}{2n} \;<\; \varepsilon
\qquad (1 \le j \le 2n).
\]

Now, let $Q_j$ be the translate of $P$ placing the point $x \in \operatorname{int}(P)$ fixed above at $V_j$. Since $x$ lies within $\varepsilon \le \varepsilon_j$ of $2\Lambda p_j\vec{d}_j$, Proposition 17 places the head of $Q_j$ in $\pi^+_{1,j,2kp_j}$, and $Q_{j+1} = Q_j + 2\Lambda p_j\vec{a}_j$ by construction, so Conditions 1--3 hold. \vspace{0.05in}

Finally, note $$\max_j\|X_j\| \le (2n-1)\varepsilon/2n,$$ so translating the entire configuration by any vector $\vec{u}$ with $\|\vec{u}\| < \varepsilon/2n$ keeps every vertex inside its $\varepsilon_j$-ball and yields another periodic necklace. The periodic points of $T$ therefore contain an open disk, and hence have positive measure in the plane.
\end{proof}

\section{Acknowledgements}

I express my sincere gratitude to my mentor, Prof. Maxim Arnold, for his invaluable patience, dedication, and feedback throughout this project. I am also grateful to him for providing the topic of this paper. I also thank the MIT PRIMES-USA program for making this research possible.

\appendix
\section{Motivating Refractive Outer Billiards}\label{sec:motivation}

The outer billiards system gets its other name, \textit{dual billiards}, from the \textit{projective duality} that it shares with classical billiards. To illustrate this duality, define the outer billiards system on $S^2,$ the 2-sphere.

There exists a natural correspondence between the set of oriented great circles and the set of points on the sphere given by a ``projective'' duality: every oriented great circle corresponds to its \textit{north pole}. Given an oriented curve $\gamma \subset S^2,$ we define its \textit{dual curve} $\gamma^*$ as follows. 

The curve $\gamma$ defines a one-parameter family of oriented tangent lines. Extend each such tangent line to obtain a corresponding family of oriented great circles. Then, the dual curve $\gamma^*$ is the collection of north poles of these great circles. Note that for great circles $a, b$ and corresponding north poles $A, B,$ the spherical distance $AB$ equals the angle between the lines $a$ and $b$. See Figure \ref{fig:greatcircle}.

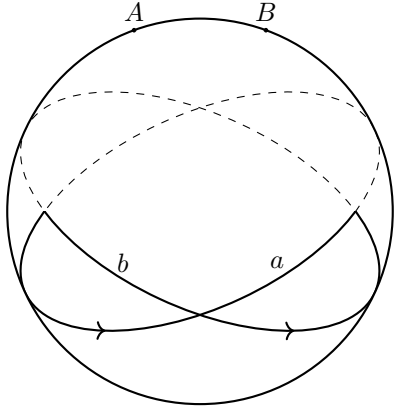
\begin{figure}[!hbt]
    \centering
\begin{tikzpicture}[scale=1.7]
  \pgfmathsetmacro{\R}{1.5}
  \tikzset{
    greatfrontA/.style={
      thick,
      postaction={decorate},
      decoration={markings,mark=at position 0.6 with {\arrow{>}}}
    },
    greatfrontB/.style={
      thick,
      postaction={decorate},
      decoration={markings,mark=at position 0.9 with {\arrow{>}}}
    }
  }
  \draw[thick] (0,0) circle (\R);
  \pgfmathsetmacro{\phiA}{25}
  \begin{scope}
    \clip (0,0) circle (\R);
    \clip (-2,0) rectangle (2,2); 
    \draw[dashed] (0,0)
      ellipse [x radius=\R, y radius=0.5*\R, rotate=\phiA];
  \end{scope}
  \begin{scope}
    \clip (0,0) circle (\R);
    \clip (-2,-2) rectangle (2,0); 
    \draw[greatfrontA] (0,0)
      ellipse [x radius=\R, y radius=0.5*\R, rotate=\phiA];
  \end{scope}
  \node at (0.6,-0.4) {$a$};
  \pgfmathsetmacro{\phiB}{-25}
  \begin{scope}
    \clip (0,0) circle (\R);
    \clip (-2,0) rectangle (2,2);
    \draw[dashed] (0,0)
      ellipse [x radius=\R, y radius=0.5*\R, rotate=\phiB];
  \end{scope}
  \begin{scope}
    \clip (0,0) circle (\R);
    \clip (-2,-2) rectangle (2,0);
    \draw[greatfrontB] (0,0)
      ellipse [x radius=\R, y radius=0.5*\R, rotate=\phiB];
  \end{scope}
  \node at (-0.6,-0.4) {$b$};
  \pgfmathsetmacro{\angA}{110}
  \pgfmathsetmacro{\angB}{70}

  \coordinate (A) at ({\R*cos(\angA)},{\R*sin(\angA)});
  \coordinate (B) at ({\R*cos(\angB)},{\R*sin(\angB)});

  \fill (A) circle (0.018);
  \fill (B) circle (0.018);

  \node[above] at (A) {$A$};
  \node[above] at (B) {$B$};

\end{tikzpicture}
\caption{The projective duality.}
\label{fig:greatcircle}
\end{figure}

This operation is indeed a duality: one can show that $\gamma^*$ is obtained from $\gamma$ by moving each point a distance of $\pi/2$ in the direction orthogonal to $\gamma,$ and that $(\gamma^*)^*$ is the antipodal curve of $\gamma$ on $S^2.$ Note also the following property: if a north pole $B$ is on the line $a$, the north pole $A$ lies on the line $b$.

Consider a billiard reflection inside $\gamma$, where a ray $a$ reflects off the tangent line $\ell$ at a point $P$ and travels in a new direction $b$. The three lines $a, b, \text{ and } \ell$ all pass through $P,$ so that all three points $A, B, \text{ and } L$ lie on the line $p$. Since $L$ is also on $\gamma^*$, we get $L = p \cap \gamma^*$. Moreover, since the angle of incidence is equal to the angle of reflection, we have $AL = LB$ from the length-angle duality. Therefore, $A, B,$ and $L$ all lie on the same line with $AL = LB$. 

The Birkhoff billiard map taking $a$ to $b$ after a reflection at $P \in \gamma$ corresponds exactly to the outer billiards map taking $A$ to $B$ after reflection about $P^*\in \gamma^*$. This duality holds only in $S^2$: in the plane, there is no direct relation between the systems. For more details, see \cite{Tabachnikov}.

Extending this duality to \textit{refractive inner billiards} defined in \cite{Arnold} gives refractive \textit{outer} billiards. The name stems from the refraction phenomena in optics: for two materials with refractive indices $n$ and $\tilde{n}$, Snell's law states that
\[
\frac{n}{\tilde{n}} = \frac{\sin{\tilde\theta}}{\sin{\theta}},
\] 
where $\theta$ is the angle of incidence and $\tilde{\theta}$ is the angle of refraction. 

Consider a billiard ball moving inside a table $\Gamma$. When the ball collides with the boundary $\partial\Gamma$, the refractive billiards system refracts the ball instead of reflecting it. We are given refractive indices $\lambda_1, \dots,\lambda_k,$ with $\lambda_1\dots\lambda_k = 1.$ For the $i$th refraction, ${\sin{\tilde{\theta}_i}}/{\sin{\theta}_i} = \lambda_i$, where $\theta_i$ is the $i^\text{th}$ angle of incidence and $\tilde{\theta}_i$ the $i^\text{th}$ angle of refraction. Then, we reflect the ray exiting the table back inside. The same projective duality motivates refractive outer billiards.

\begin{thebibliography}{9}


\bibitem{Gutkin}
E. Gutkin, N. Simanyi,
Dual polygonal billiards and necklace dynamics,
{\it Comm. Math. Phys.} {\bf 143} (1992), 431--449.

\bibitem{Culter}
S. Tabachnikov,
A proof of Culter's theorem on the existence of periodic orbits in polygonal outer billiards,
{\it Geom. Dedicata} {\bf 129} (2007), 83--87.

\bibitem{Arnold}
M. Arnold, J. Park, 
Snell meets Fagnano. Path optimization through an imperfect mirror,
available online at the URL: \url{https://arxiv.org/abs/2512.02236}.

\bibitem{Kolodziej}
R. Kolodziej,
The antibilliard outside a polygon,
{\it Bull. Polish Acad. Sci. Math.} {\bf 37} (1989), 163--168.

\bibitem{Moser}
J. Moser,
Is the solar system stable?,
{\it Math. Intelligencer} {\bf 1} (1978), 65--71.

\bibitem{Moser2}
J. Moser,
{\it Stable and random motions in dynamical systems},
Annals of Mathematics Studies, vol.~77,
Princeton Univ. Press, 1973.

\bibitem{Vivaldi}
F. Vivaldi, A. V. Shaidenko,
Global stability of a class of discontinuous dual billiards,
{\it Comm. Math. Phys.} {\bf 110} (1987), 625--640.

\bibitem{Genin}
D. Genin,
{\it Regular and chaotic dynamics of outer billiards},
Ph.D. thesis, Pennsylvania State University, 2005.

\bibitem{Dolgopyat}
D. Dolgopyat, B. Fayad,
Unbounded orbits for semicircular outer billiard,
{\it Ann. Henri Poincaré} {\bf 10} (2009), 357--375.

\bibitem{Schwarz}
R. E. Schwartz,
Unbounded orbits for outer billiards. I,
{\it J. Mod. Dyn.} {\bf 1} (2007), 371--424.

\bibitem{Tabachnikov}
S. Tabachnikov,
{\it Geometry and billiards},
Amer. Math. Soc., 2005.

\bibitem{Boyland}
P. Boyland, Dual Billiards, twist maps, and impact oscillators, {\it Nonlinearity} {\bf 9} (1996), 1411--1438.

\bibitem{Douady}
R. Douady, These de 3-eme cycle, Universite de Paris 7, 1982
\end{thebibliography}
\end{document}